\documentclass[12pt]{amsart}
\usepackage{hyperref}
\usepackage{amssymb,amsmath,amsthm,mathrsfs}
\allowdisplaybreaks
\numberwithin{equation}{section}
\newtheorem{thm}{Theorem}[section]
\newtheorem{lemma}[thm]{Lemma}
\newtheorem{corollary}[thm]{Corollary}
\newtheorem{remark}[thm]{Remark}

\newtheorem{definition}[thm]{Definition}

\renewcommand{\(}{\left(}
\renewcommand{\)}{\right)}
\newcommand{\lj}{\left|}
\newcommand{\rj}{\right|}
\newcommand{\E}{{\rm E}}
\newcommand{\tr}{{\rm tr}}
\newcommand{\mb}{\mathbf}
\newcommand{\hi}{H\"older's inequality}

\begin{document}
\title[SPECTRUM CONVERGENCE RATES OF QUATERNION MATRICES]{CONVERGENCE RATES OF THE SPECTRAL DISTRIBUTIONS OF LARGE RANDOM QUATERNION SELF-DUAL HERMITIAN MATRICES}

\author{Yanqing Yin,\ \ Zhidong Bai}
\thanks{Y. Q. Yin was partially supported by a grant CNSF 11301063; Z. D. Bai was partially supported by CNSF  11171057 and PCSIRT
}

\address{KLASMOE and School of Mathematics \& Statistics, Northeast Normal University, Changchun, P.R.C., 130024.}
\email{yinyq799@nenu.edu.cn}
\address{KLASMOE and School of Mathematics \& Statistics, Northeast Normal University, Changchun, P.R.C., 130024.}
\email{baizd@nenu.edu.cn}


\maketitle
\begin{abstract}
In this paper, convergence rates of the spectral distributions of quaternion self-dual Hermitian matrices are investigated. We show that under conditions of finite 6th moments, the expected spectral distribution of a large quaternion self-dual Hermitian matrix converges to the semicircular law in a rate of $O(n^{-1/2})$ and the spectral distribution itself converges to the semicircular law in rates $O_p(n^{-2/5})$ and $O_{a.s.}(n^{-2/5+\eta})$. Those results include \emph{GSE} as a special case.
\end{abstract}

\section{Introduction and Main results}
Research on the Limiting Spectral Analysis (LSA) of Large Dimensional Random Matrices (LDRM) has attracted considerable interest among mathematicians, probabilitists and statisticians since modern computers are widely applied to very disciplines in 70's of last centure. In nuclear physics, the energy levels are described by the eigenvalues of the Hamiltonian, which can be formulated as an Hermitian matrix of very large order. In accordance with the consequences of time-reversal invariance, three kinds of ensembles were proposed:
 \begin{itemize}
 \item orthogonal ensemble: the set of $n\times n$ random real symmetric matrices,
 \item unitary ensemble:  the set of $n\times n$ random complex Hermitian matrices,
 \item symplectic ensemble: the set of $n\times n$ random quaternion self-dual Hermitian matrices.
 \end{itemize}
For every single matrix in those ensembles, the entries on and above the diagonal are usually assumed to be independent. A matrix in unitary ensemble (including orthogonal ensemble) is also called a Wigner matrix.

Here, for readers' convenience, we shall give a short introduction to the quaternion self-dual Hermitian matrices. Quaternions were invented in 1843 by the Irish mathematician William Rowen Hamilton as an extension of complex numbers into three dimensions \cite{hamilton1866elements,1973},
and it's well known that the quaternion field $\mathbb{Q}$ can be represented as a two-dimensional complex vector space \cite{chevalley1946lie}.
 Thus this representation associates any $n\times n$ quaternion matrix with a $2n\times2n$ complex matrix.
Let $i=\sqrt{-1}$ and define three $2\times 2$ matrices:
$$
\mb{i}_1=\left(
              \begin{array}{cc}
                i & 0 \\
                0 & -i \\
              \end{array}
            \right), \quad
\mb{i}_2=\left(
              \begin{array}{cc}
                0 & 1 \\
                -1 & 0 \\
              \end{array}
            \right), \quad
\mb{i}_3=\left(
              \begin{array}{cc}
                0 & i \\
                i & 0 \\
              \end{array}
            \right).
$$
Note that in what follows, we use $\mb I_n$ to denote $n \times n$ unit matrix.
We can verify that:
 \begin{gather*}
\mathbf{i}_1^2=\mathbf{i}_2^2=\mathbf{i}_3^2=-\mathbf{I}_2,\quad \mathbf{i}_1 =\mathbf{i}_2 \mathbf{i}_3 =-\mathbf{i}_3 \mathbf{i}_2 ,\\
 \mathbf{i}_2 =\mathbf{i}_3 \mathbf{i}_1 =-\mathbf{i}_1 \mathbf{i}_3 ,\quad \mathbf{i}_3 =\mathbf{i}_1 \mathbf{i}_2 =-\mathbf{i}_2 \mathbf{i}_1 .
\end{gather*}
For any real numbers $a,b,c,d,$ let
$$q=a\mathbf{I}_2+b\mathbf{i}_1+c\mathbf{i}_2+d\mathbf{i}_3=\left(
                                                      \begin{array}{cc}
                                                        a+bi & c+di \\
                                                        -c+di & a-bi \\
                                                      \end{array}
                                                    \right)=\left(
                                                              \begin{array}{cc}
                                                                \alpha & \beta \\
                                                             -\bar{\beta} & \bar{\alpha} \\
                                                              \end{array}
                                                            \right),
$$
then $q$ is a quaternion. The quaternion conjugate of $q$ is defined by
$$\bar{q}=a\mathbf{I}_2-b\mathbf{i}_1-c\mathbf{i}_2-d\mathbf{i}_3=\left(
                                                      \begin{array}{cc}
                                                        a-bi & -c-di \\
                                                        c-di & a+bi \\
                                                      \end{array}
                                                    \right)=\left(
                                                              \begin{array}{cc}
                                                                \bar{\alpha} & -\beta \\
                                                             \bar{\beta} & \alpha \\
                                                              \end{array}
                                                            \right),
$$
and the quaternion norm of $q$ is defined by
$$\| q \| = \sqrt {{a^2} + {b^2} + {c^2} + {d^2}}  = \sqrt {{{\left| \alpha  \right|}^2} + {{\left| \beta  \right|}^2}}. $$

An $n\times n$ quaternion self-dual Hermitian matrix $^q\mathbf{A}_n=(x_{jk})_{n\times n}$ is a matrix whose entries $x_{jk} \ (j,k=1,\cdots,n)$ are quaternions and satisfy $x_{jk}=\bar{x}_{kj}$. If each $x_{jk}$ is viewed as a $2\times 2$ matrix, then $\mathbf A_n=(x_{jk})$ is in fact a $2n\times2n$ Hermitian matrix.
And, the entries of $\mathbf A_n$ can be represented as
$$x_{jk} = \left( {\begin{array}{*{20}{c}}
a_{jk} + b_{jk}i & c_{jk} + d_{jk}i \\
-{ c }_{jk}+{ d }_{jk}i&{a }_{jk}-{ b }_{jk}i
\end{array}} \right)
=\left( {\begin{array}{*{20}{c}}
\alpha_{jk} &\beta_{jk} \\
{ - \overline \beta }_{jk}&{\overline \alpha }_{jk}
\end{array}} \right), 1\le j<k\le n, $$
 and
$x_{jj} = \left( {\begin{array}{*{20}{c}}
a_{jj} & 0 \\
0& a_{jj}
\end{array}} \right),$ where $ a_{jk},b_{jk},c_{jk},d_{jk} \in \mathbb{R} \ {\rm and} \ 1\le j,k \le n.$
It is well known (see \cite{zhang1997quaternions}) that the multiplicities of all the eigenvalues of $\mathbf A_n$ are even.

Suppose that an $n\times n$ matrix $\mathbf H_n$ is Hermitian whose eigenvalues are denoted by  $\lambda_1\geq \lambda_2\geq\cdots\geq \lambda_n$.
Then the empirical spectral distribution (ESD) of $\mathbf H_n$ is defined as ${F^{\mathbf{H_n}}}(x) = \frac{1}{n}\sum_{i = 1}^n {I({\lambda _i} \le x)},$
where ${I(\cdot})$ is the indicator function.  In \cite{wigner1955, wigner1958}, Wigner proved that the expected ESD of a real Wigner matrix tends to the so-called semicircular law with density function: 
$$f(x)=\frac{1}{2\pi\sigma^2}\sqrt{4\sigma^2-x^2},~~x\in[-2\sigma,2\sigma],$$
where $\sigma^2$ is the scale parameter. Then, Juh\'{a}sz \cite{juhasz1978spectrum} and F\"{u}redi and Koml\'{o}s \cite{Furedi1981} firstly studied the weak limit of the extreme eigenvalues for a Wigner matrix.
 In 1988, Bai and Yin \cite{BaiYin1993} got the necessary and sufficient conditions for the strong convergence of the extreme eigenvalues.
And in \cite{tracy1994level}, Tracy and Widom derived the limiting distribution of the largest eigenvalue of \emph{Gaussian Orthogonal Ensemble (GOE)},
\emph{Gaussian Unitary Ensemble (GUE)}, and \emph{Gaussian Symplectic Ensemble (GSE)}. Much more new results can be found in \cite{o2013universality,tao2011wigner,Wang2012,PhysRevE.77.041108,Dumitriu2008} and references therein.
Another important problem is the convergence rate of the ESD of LDRM. The main method to establish such rates is the use of Stieltjes transform. In 1993, Bai \cite{Bai1993a} established an inequality and got the convergence rate of the expected spectral distribution of a Wigner matrix. For later results, we refer to Bai \cite{Bai199795,Bai1999,Bai2002} and F. G$\ddot{\rm o}$tze \cite{gotze2003rate,gotze2003}.

In \cite{yinbai2013}, the semicircular law for quaternion self-dual Hermitian matrices was proved. In \cite{yinbai2013e}, the asymptotic properties of the extreme eigenvalues for random quaternion self-dual Hermitian matrices was investigated. In this paper, we shall use this two papers' result and Bai inequality to establish a convergence rate of the expected spectral distribution of the quaternion self-dual Hermitian matrices. Our main theorems can be described as following:

\begin{thm}\label{mth1}
Suppose that $\mb S_n:=\frac{1}{{\sqrt n }}{\mb Y_n}= \frac{1}{{\sqrt n }} \({y_{jk}}\)_{n\times n}$, where $y_{kj}^*=y_{jk}=a_{jk}\mb I_2+b_{jk} \mb i_1+c_{jk} \mb i_2+d_{jk} \mb i_3$, is a quaternion self-dual Hermitian matrix whose entries above and on the diagonal are independent and satisfy: \begin{itemize}
  \item[(i)] ${\rm E} y_{jk}=\left(
                                     \begin{array}{cc}
                                       0 & 0 \\
                                       0 & 0 \\
                                     \end{array}
                                   \right)
  ,\mbox{ for all }1\le j\le k\le n.$
  \item[(ii)] ${\rm E} \|y_{jj}\|^2=\sigma^2\leq M<\infty, ~{\rm E} \|y_{jk}\|^2=1, \mbox{ for all }1\le j< k\le n.$
  \item[(iii)] $\sup_{1\leq j<k\leq n}\E\|y_{jk}\|^6,\E\|y_{jj}\|^3\leq M <\infty.$
\end{itemize}
Let $F$ and $F_n^s$ denote the semicircular law and the ESD of $\mb S_n$, respectively. Then we have as $n\to\infty$, for any $\eta>0$
\begin{align}
\begin{cases}
\|\E F_n^s-F\|_d=O(n^{-1/2}),\\
\|F_n^s-F\|_d=O_p(n^{-2/5}),\\
\|F_n^s-F\|_d=O_{a.s.}(n^{-2/5+\eta}),
\end{cases}
\end{align}
where $\|F\|_d:=\sup \limits_{x} |F(x)|$.
\end{thm}
\begin{remark}
  We note that the entries of $\mb S_n$ may depend on $n$, but the index $n$ is suppressed for brevity.
\end{remark}

The proof of Theorem \ref{mth1} basically depends on the proof of the following theorem and an estimate of the difference between ESD's of $\mb S_n$ and $\mb W_n$ defined in Theorem \ref{mth2}. 

\begin{thm}\label{mth2}
Under the conditions of Theorem \ref{mth1}, let $\sigma_{jk}^2={\rm var}\big(y_{jk}I(\|y_{jk}\|\leq n^{1/4})\big)$ and ${\bf W}_n= \frac{1}{{\sqrt n }} \({x_{jk}}\)_{n\times n}$, where
\begin{align}
x_{jk}=
\begin{cases}
\sigma_{jk}^{-1}\big(y_{jk}I(\|y_{jk}\|\leq n^{1/4})-\E y_{jk}I(\|y_{jk}\|\leq n^{1/4})\big) &\mbox {$j\neq k$},\\
\sigma\sigma_{jj}^{-1}\big(y_{jj}I(\|y_{jj}\|\leq n^{1/4})-\E y_{jj}I(\|y_{jj}\|\leq n^{1/4})\big) &\mbox{$j=k$}.
\end{cases}
\end{align}
Let $F$ and $F_n^w$ denote the semicircular law and the ESD of $\mb W_n$, respectively. Then we have as $n\to\infty$, for any $\eta>0$
\begin{align}
\begin{cases}
\|\E F_n^w-F\|_d=O(n^{-1/2}),\\
\|F_n^w-F\|_d=O_p(n^{-2/5}),\\
\|F_n^w-F\|_d=O_{a.s.}(n^{-2/5+\eta}).
\end{cases}
\end{align}
\end{thm}

\section{Main Tools and some lemmas}
Before proving our theorems, we present some lemmas which are needed in the proofs of the main theorems.

To begin with, we introduce some notation.
For any function of bounded variation $G$ on the real line, its Stieltjes transform is defined by
$$s_G(z)=\int\frac{1}{y-z}dG(y),~~z\in\mathbb{C}^{+}\equiv\{z\in\mathbb{C}:\Im z>0\}.$$
Let $\mb Q_k=(x'_{1k}, \ldots, x'_{(k-1)k}, x'_{(k+1)k}, \ldots, x'_{nk})'=(\varpi'_k, \varrho'_k)'_{(2n-2)\times 2}$ denote the $k$-th 
quaternion column of $\mb W_n$ with $k$-th quaternion elements removed. Let $\mb W_n(k)$ be the matrix obtained from $\mb W_n$ with the $k$-th quaternions column and
the $k$-th quaternions row removed. Let $s(z)$ be the Stieltjes transform of semicircular law and $s_n(z)$ be the Stieltjes transform of $F_n^w$. Moreover, write
$$\mb D=\(\mb W_n-z\mb I_{2n}\)^{-1}, \mb D_k=\(\mb W_n(k)-z\mb I_{2n-2}\)^{-1},{\mb P}_k=\(\mb W_n(k)-z\mb I_{2n-2}\)^{-2},$$
$$\Delta=\|\E F_n^w-F\|_d, \quad t_n(z)=\(z+\E s_n(z)\)^{-1},
$$
$$\epsilon_k=n^{-1/2}x_{kk}-n^{-1}\mb Q_k^*\mb D_k\mb Q_k+\E s_n(z)\cdot\mb I_2, \xi_k(z)=\(\(z+\E s_n(z)\)\cdot\mb I_2-\epsilon_k\)^{-1}.$$
Last but not the least, we remind the reader that throughout this paper, a local constant C may take different value at different appearances.

\subsection{Some auxiliary lemmas: Part I}
\begin{lemma}[Theorem 2.2 in \cite{Bai1993a}]\label{bai93in}
Let $F$ be a distribution function and $G$ be a function of bounded variation satisfying $\int |F(x)-G(x)|\mathrm{d}x< \infty.$ Denote their Stieltjes transforms by $f(z)$ and $g(z)$, respectively, where $z=u+iv \in \mathbb{C^+}$. Then we have
\begin{align}
  \|F-G\|_d&:=\sup \limits_{x} |F(x)-G(x)| \\ \notag
  &\leq \frac{1}{\pi(1-\kappa)(2\gamma-1)} {\bigg [} \int_{-A}^{A}|f(z)-g(z)|\mathrm{d}u \\ \notag
  &+2\pi v^{-1} \int_{|x|>B}|F(x)-G(x)|\mathrm{d}x \\ \notag
  &+v^{-1} \sup \limits_{x} \int_{|y|\leq 2va}|G(x+y)-G(x)|\mathrm{d}y \bigg],
\end{align}
where $a$ and $\gamma$ are constants related to each other by $\gamma=\frac{1}{\pi}\int_{|u|<a}\frac{1}{u^2+1}\mathrm{d}u>\frac{1}{2},$ and $A$ and $B$ are positive constants such that $A>B$ and $\kappa=\frac{4B}{\pi(A-B)(2\gamma-1)}<1.$
\end{lemma}

\begin{definition} A matrix is called Type-\uppercase\expandafter{\romannumeral1} matrix if it has the following structure:
\[\left( {\begin{array}{*{20}{c}}
{{t_1}}&0&{{a_{12}}}&{{b_{12}}}& \cdots &{{a_{1n}}}&{{b_{1n}}}\\
0&{{t_1}}&{{c_{12}}}&{{d_{12}}}& \cdots &{{c_{1n}}}&{{d_{1n}}}\\
{{d_{12}}}&{ - {b_{12}}}&{{t_2}}&0& \cdots &{{a_{2n}}}&{{b_{2n}}}\\
{ - {c_{12}}}&{{a_{12}}}&0&{{t_2}}& \cdots &{{c_{2n}}}&{{d_{2n}}}\\
 \vdots & \vdots & \vdots & \vdots & \ddots & \vdots & \vdots \\
{{d_{1n}}}&{ - {b_{1n}}}&{{d_{2n}}}&{ - {b_{2n}}}& \cdots &{{t_n}}&0\\
{ - {c_{1n}}}&{{a_{1n}}}&{ - {c_{2n}}}&{{a_{2n}}}& \ldots &0&{{t_n}}
\end{array}} \right).\]
Here all the  entries are  complex numbers.
\end{definition}
\begin{definition}A matrix is called Type-\uppercase\expandafter{\romannumeral2} matrix if it has the following structure:
\begin{footnotesize}
\[\left( {\begin{array}{*{20}{c}}
{{t_1}}&0&{{a_{12}} + {c_{12}}  i}&{{b_{12}} + {d_{12}}  i}& \cdots &{{a_{1n}} + {c_{1n}}  i}&{{b_{1n}} + {d_{1n}}  i}\\
0&{{t_1}}&{ - {{\bar b}_{12}} - {{\bar d}_{12}}  i}&{{{\bar a}_{12}} + {{\bar c}_{12}}  i}& \cdots &{ - {{\bar b}_{1n}} - {{\bar d}_{1n}}  i}&{{{\bar a}_{1n}} + {{\bar c}_{1n}}  i}\\
{{{\bar a}_{12}} + {{\bar c}_{12}}  i}&{ - {b_{12}} - {d_{12}}  i}&{{t_2}}&0& \cdots &{{a_{2n}} + {c_{2n}}  i}&{{b_{2n}} + {d_{2n}}  i}\\
{{{\bar b}_{12}} + {{\bar d}_{12}}  i}&{{a_{12}} + {c_{12}}  i}&0&{{t_2}}& \cdots &{ - {{\bar b}_{2n}} - {{\bar d}_{2n}}  i}&{{{\bar a}_{2n}} + {{\bar c}_{2n}}  i}\\
 \vdots & \vdots & \vdots & \vdots & \ddots & \vdots & \vdots \\
{{{\bar a}_{1n}} + {{\bar c}_{1n}}  i}&{ - {b_{1n}} - {d_{1n}}  i}&{{{\bar a}_{2n}} + {{\bar c}_{2n}}  i}&{ - {b_{2n}} - {d_{2n}} i}& \cdots &{{t_n}}&0\\
{{{\bar b}_{1n}} + {{\bar d}_{1n}}  i}&{{a_{1n}} + {c_{1n}}  i}&{{{\bar b}_{2n}} + {{\bar d}_{2n}}  i}&{{a_{2n}} + {c_{2n}}  i}& \ldots &0&{{t_n}}
\end{array}} \right).\]
\end{footnotesize}
Here $i=\sqrt{-1}$ denotes the usual imaginary unit and all the other variables are  complex numbers.
\end{definition}

\begin{lemma}[Theorem 1.1 in \cite{yinbai2013}]\label{yin2013}
For all $n\geq1$, if a complex  matrix  $\Omega_n$ is  invertible and of Type-$I\!I$, then $\Omega_n^{-1}$ is a Type-\uppercase\expandafter{\romannumeral1} matrix.
\end{lemma}

\begin{corollary}\label{form}
By Lemma \ref{yin2013}, if the conditions of Theorem \ref{mth2} hold, we have
\begin{enumerate}
  \item $\mb D$ and $\mb D_k$ are all Type-\uppercase\expandafter{\romannumeral1} matrices ,
  \item $\xi_k(z)$ and $\epsilon_k$ are all scalar matrices, where $$\epsilon_k=\left(
                                                                                 \begin{array}{cc}
                                                                                   \E s_n(z)-n^{-1}\varpi_k^*{\mb D}_k\varpi_k & 0 \\
                                                                                   0 & \E s_n(z)-n^{-1}\varrho_k^*{\mb D}_k\varrho_k \\
                                                                                 \end{array}
                                                                               \right)+n^{-1/2}x_{kk},
  $$
\end{enumerate}
\end{corollary}

\begin{lemma}[Corollary A.41 in \cite{bai2010spectral}]\label{A.41}
Let $A$ and $B$ be two $n \times n$  normal matrices with their ESDs ${F^A}$ and ${F^B}$.  Then, $${L^3}({F^A}, {F^B}) \le \frac{1}{n}{\rm tr}[(A - B){(A - B)^ * }], $$
where $L(\cdot,\cdot)$ is the Levy distance between two distributions (See Remark A.39. in \cite{bai2010spectral}).
\end{lemma}

\begin{remark}
In view of Lemma B.18 of \cite{bai2010spectral} and noticing that the semicircular law satisfies the Lipschitz condition, the metric $L(F_n^w,F)$ dominates the Kolmogorov distance $\|F_n^w-F\|_d$.
\end{remark}

\begin{lemma}[Theorem A.43 in \cite{bai2010spectral}]\label{A.43}
Let $\mb A$ and $\mb B$ be two $p \times n$ Hermitian matrices.  Then,  $$\left\| {{F^{\mb A}} - {F^{\mb B}}} \right\|_d \le \frac{1}{n}{\rm rank}({\mb A} - {\mb B}). $$
\end{lemma}

\begin{lemma}[Theorem A.13 in \cite{bai2010spectral}]\label{A.13}
Let ${\mathbf A}=\(a_{jk}\)_{j,k=1}^{n}$ be a complex matrix and $f$ be an increasing and convex function. Then we have
$$\sum_{j=1}^nf\(\left|a_{jj}\right|\)\leq \sum_{j=1}^nf(s_j\({\mathbf A}\)).$$
where $s_j\({\mathbf A}\)$ denote the singular values of ${\mathbf A}$. When ${\mathbf A}$ is Hermitian, $s_j\({\mathbf A}\)$ can be replaced by eigenvalues and $f$ need not to be increasing.
\end{lemma}

\begin{lemma}[Lemma 2.12 in \cite{bai2010spectral}]\label{2.12}
Let $\{{\tau_k}\}$ be a complex martingale difference sequence with respect to the increasing $\sigma$-fields $\mathscr{F}_k$. Then, for $p > 1, $ ${\rm E}{\left| {\sum {{\tau_k}} } \right|^p} \le {K_p}{\rm E}{({\sum {\left| {{\tau_k}} \right|} ^2})^{p/2}}.$
\end{lemma}

\begin{lemma}[Lemma 2.13 in \cite{bai2010spectral}]\label{2.13}
Let $\{{\tau_k}\}$ be a complex martingale difference sequence with respect to the increasing $\sigma$-fields $\mathscr{F}_k$, and let ${\rm E}_{k}$ denote conditional expectation with respect to $\mathcal{F}_k$. Then ,  for $p\geq2,$ $${\rm E}{\left| {\sum {{\tau_k}} } \right|^p} \le {K_p}\({\rm E}{\({\sum {\rm E}_{k-1}{\left| {{\tau_k}} \right|} ^2}\)^{p/2}}+{\rm E}{\sum}{\left| {{\tau_k}} \right|} ^p\).$$
\end{lemma}

\begin{lemma}[Theorem 2.3 in \cite{dilworth1993some}]\label{B.27}
Let $\mathscr{F}_k$ be a sequence of increasing $\sigma$-fields and $\{{\tau_k}\}$ be a sequence of integrable random variables. Then, for any $1\leq q\leq p < \infty$ , we have
$$\E \(\sum_{k=1}^\infty \left|\E\(\tau_k|\mathscr{F}_k\)\right|^q\)^{p/q}\leq \(\frac{p}{q}\)^{p/q}\E\(\sum_{k=1}^\infty \left|\tau_k\right|^q\)^{p/q}.$$
\end{lemma}

\begin{lemma}[See appendix A.1.4 in \cite{bai2010spectral}]\label{inv}
Suppose that the matrix $\Sigma $ is nonsingular and has the partition as given by $\left( {\begin{array}{*{20}{c}}
{{\Sigma _{11}}}&{{\Sigma _{12}}}\\
{{\Sigma _{21}}}&{{\Sigma _{22}}}
\end{array}} \right),$ then, if $\Sigma_{11}$ is  nonsingular, the inverse of $\Sigma $ has the form $$\left( {\begin{array}{*{20}{c}}
{\Sigma _{11}^{ - 1} + \Sigma _{11}^{ - 1}\Sigma _{12}^{}\Sigma _{22. 1}^{ - 1}\Sigma _{21}^{}\Sigma _{11}^{ - 1}}&{ - \Sigma _{11}^{ - 1}\Sigma _{12}^{}\Sigma _{22. 1}^{ - 1}}\\
{ - \Sigma _{22. 1}^{ - 1}\Sigma _{21}^{}\Sigma _{11}^{ - 1}}&{\Sigma _{22. 1}^{ - 1}}
\end{array}} \right)$$
where $\Sigma _{22. 1}^{ - 1} = \Sigma _{22}^{} - \Sigma _{21}^{}\Sigma _{11}^{ - 1}\Sigma _{12}^{}$.
\end{lemma}
\begin{lemma}[(A.1.12) in \cite{bai2010spectral}]\label{A.1.12}
Let $z = u + iv, v > 0, $ and let $\mb A$ be an $n \times n$ Hermitian matrix.  ${{\mb A}_k}$ be the k-th major sub-matrix of $\mb A$ of order $(n-1)$, to be the matrix resulting from the $k$-th row and column from $\mb A$.  Then
$$\left| {{\rm tr}{{(\mb A - z{I_n})}^{ - 1}} -{ \rm tr}{{({{\mb A}_k} - z{I_{n - 1}})}^{ - 1}}} \right| \le v^{-1}.$$
\end{lemma}

\begin{lemma}[Lemma 2.11 in \cite{bai2010spectral}]\label{semst}
Let $z = u + iv, v > 0, $ $s(z)$ be the Stieltjes transform of the semicircular law.  Then,  we have $s(z) =  - \frac{1}{2}(z - \sqrt {{z^2} - 4} )$.
\end{lemma}

\begin{lemma}[Lemma B.22 in \cite{bai2010spectral}]\label{B.22}
Let $G$ be a function of bounded variation. Let $g(z)$ denote its Stieltjes transform. When $z = u + iv,$ with  $v>0,$ we have
$$\sup_u \lj g(z)\rj\leq \pi v^{-1}\|G\|_d.$$
\end{lemma}

\begin{lemma}\label{sz}
Let $s(z)$ be the Stieltjes transform of the semicircular law, which is given by Lemma \ref{semst}. We have $\lj s(z)\rj<1.$
\end{lemma}
\begin{proof}
 Since $s(z)\(-\frac{1}{2}\(z+\sqrt{z^2-4}\)\)=1$, we shall complete the proof by noticing that both the real and imaginary parts of $z$ and $\sqrt{z^2-4}$ have the same signs.
\end{proof}\

\subsection{Some auxiliary lemmas: Part II}
\begin{lemma}\label{B.26}
Let ${\mathbf A}=\(a_{jk}\)_{j,k=1}^{2n}$ be a $2n\times2n$ non-random matrix and ${\mathbf X}=(x_1',\cdots,x_n')'$ be a random quaternion vector of independent entries, where for $1\leq j\leq n$,
 $x_j=\left(
        \begin{array}{cc}
          e_j+f_j\cdot i & c_j+d_j\cdot i\\
          -c_j+d_j\cdot i & e_j-f_j\cdot i \\
        \end{array}
      \right)
 =\left(
        \begin{array}{cc}
          \alpha_j & \beta_j \\
          -\bar \beta_j & \bar \alpha_j \\
        \end{array}
      \right)
$.
Assume that $\E x_j=\left(
                      \begin{array}{cc}
                        0 & 0 \\
                        0 & 0 \\
                      \end{array}
                    \right)
$, $\E\left\|x_j\right\|^2=1$, and $\E\left\|x_j\right\|^l\leq \phi_l$. Then, for any $p\geq 1$, we have
$$\E\left|\tr{\mathbf X}^*{\mathbf A}{\mathbf X}-\tr {\mathbf A} \right|^p \leq C_p \(\(\phi_4\tr\({\mathbf A}{\mathbf A^*}\)\)^{p/2}+\phi_{2p}\tr\({\mathbf A}{\mathbf A^*}\)^{p/2}\),$$
where $C_p$ is a constant depending on $p$ only.
\end{lemma}
\begin{proof}
At first, cut ${\mathbf A}$ into $n^2$ blocks and rewrite ${\mathbf A}=\(\tilde{a}_{jk}\)_{j,k=1}^{n}$,
 where $\tilde{a}_{jk}=\(
   \begin{array}{cc}
    a_{2j-1,2k-1} & a_{2j-1,2k} \\
       a_{2j,2k-1} & a_{2j,2k}\\
        \end{array}
         \)
$. Use the expression
\begin{equation}\label{B.26p}
  \tr{\mathbf X}^*{\mathbf A}{\mathbf X}-\tr {\mathbf A}=\sum_{j=1}^{n} \(\|{x_j}\|^2-1\)\tr \tilde{a}_{jj}+\sum_{j=1}^{n} \sum_{k=1}^{j-1}\tr\(\tilde{a}_{kj}{x_j}{x}_k^*+\tilde{a}_{jk}{x}_k x_j^*\).
\end{equation}
For $p=1$, noticing that $\E\left|\alpha_{j}\right|^2\leq 1$ and $\E\left|\beta_{j}\right|^2\leq 1$, by Lemma \ref{A.13} and Lyapunov inequality, we obtain
\begin{align*}
  &\E\left|\tr{\mathbf X}^*{\mathbf A}{\mathbf X}-\tr {\mathbf A}\right| \\
  \leq &\sum_{j=1}^{n} \E\left|\|{x_j}\|^2-1\right| \left|\tr \tilde{a}_{jj}\right|+\(\E\left|\sum_{j=1}^{n} \sum_{k=1}^{j-1}\tr\(\tilde{a}_{kj}{x_j}{x}_k^*+\tilde{a}_{jk}{x}_k x_j^*\)\right|^2\)^{1/2} \\
    \leq &C\(\sum_{j=1}^{n} \left|\tr \tilde{a}_{jj}\right|+\(\E\left|\sum_{j=1}^{n} \sum_{k=1}^{j-1}\tr\(\tilde{a}_{kj}{x_j}{x}_k^*\)\right|^2+\E\left|\sum_{j=1}^{n} \sum_{k=1}^{j-1}\tr\(\tilde{a}_{jk}{x}_k x_j^*\)\right|^2\)^{1/2}\) \\
  \leq & C \Bigg(\sum_{j=1}^n\(\left|a_{2j-1,2j-1}\right|+\left|a_{2j,2j}\right|\) \\
  & + \bigg(\sum_{j=1}^{n} \sum_{k=1}^{n}\(\left|a_{2j-1,2k-1}\right|^2+\left|a_{2j-1,2k}\right|^2+\left|a_{2j,2k-1}\right|^2+\left|a_{2j,2k}\right|^2\)\bigg)^{1/2}\Bigg) \\
  \leq & C\left[\tr\({\mathbf A}{\mathbf A}^*\)^{1/2}+\(\tr{\mathbf A}{\mathbf A}^*\)^{1/2}\right]  \leq  C_1\left[\phi_2\tr\({\mathbf A}{\mathbf A}^*\)^{1/2}+\(\phi_4\tr{\mathbf A}{\mathbf A}^*\)^{1/2}\right].
\end{align*}
Now, assume $1<p\leq2$. By Lemma \ref{2.12} and Lemma \ref{A.13}, we have

\begin{align}\label{B.26p1}
&\E\left|\sum_{j=1}^{n}\(\|{x_j}\|^2-1\)\tr \tilde{a}_{jj}\right|^p\leq C \E\(\sum_{j=1}^{n}\left|\|{x_j}\|^2-1\right|^2\left|\tr \tilde{a}_{jj}\right|^2\)^{p/2} \notag\\
&\leq C\sum_{j=1}^{n}\E\left|\|{x_j}\|^2-1\right|^p\left|\tr \tilde{a}_{jj}\right|^p
\leq C_p\phi_{2p}\tr\({\mathbf A}{\mathbf A}^*\)^{p/2}.
\end{align}

Furthermore, using the \hi, we have

\begin{align}\label{B.26p2}
  &\E\left|\sum_{j=1}^{n} \sum_{k=1}^{j-1}\tr\(\tilde{a}_{kj}{x_j}{x}_k^*+\tilde{a}_{jk}{x}_k x_j^*\)\right|^p \notag\\
  \leq & C\(\E\left|\sum_{j=1}^{n} \sum_{k=1}^{j-1}\tr\(\tilde{a}_{kj}{x_j}{x}_k^*+\tilde{a}_{jk}{x}_k x_j^*\)\right|^2\)^{p/2}
  \leq  C_p\(\phi_4\tr{\mathbf A}{\mathbf A}^*\)^{p/2}.
\end{align}

Combining (\ref{B.26p1}) and (\ref{B.26p2}), we complete the proof of the lemma for the case $1<p\leq2$.
Now, we will proceed with the proof of the lemma by induction on $p$. Assume that the lemma is true for $1\leq p\leq 2^t$, then consider the case $2^t<p\leq2^{t+1}$ with $t\geq1$.
Denote $\E_j$ be the conditional expectation given $\{x_1,\cdots,x_j\}$. By Lemma \ref{2.13} and Lemma \ref{A.13}, we obtain
\begin{align}\label{B.26p3}
&\E\left|\sum_{j=1}^{n}\(\|{x_j}\|^2-1\)\tr \tilde{a}_{jj}\right|^p \notag\\
\leq &C\(\(\sum_{j=1}^{n}\E\left|\|{x_j}\|^2-1\right|^2\left|\tr \tilde{a}_{jj}\right|^2\)^{p/2}+\sum_{j=1}^{n}\E\left|\|{x_j}\|^2-1\right|^p\left|\tr \tilde{a}_{jj}\right|^p\) \notag\\
\leq &C\(\(\phi_4\tr{\mathbf A}{\mathbf A}^*\)^{p/2}+\phi_{2p}\tr\({\mathbf A}{\mathbf A}^*\)^{p/2}\),
\end{align}
and
\begin{align}\label{B.26p4}
&\E\left|\sum_{j=1}^{n}\sum_{k=1}^{j-1}\tr \tilde{a}_{kj}x_jx_k^*\right|^p \notag\\
\leq &C\(\E\(\sum_{j=1}^{n}\E_{j-1}\left|\sum_{k=1}^{j-1}\tr\tilde{a}_{kj}x_jx_k^*\right|^2\)^{p/2}+\sum_{j=1}^n\E\left|\sum_{k=1}^{j-1}\tr\tilde{a}_{kj}x_jx_k^*\right|^p\).
\end{align}
Write $x_k^*\tilde{a}_{kj}=\left(
                                \begin{array}{cc}
                                  e_{kj} & f_{kj} \\
                                  g_{kj} & h_{kj} \\
                                \end{array}
                              \right)
$, where
\begin{align*}
&e_{kj}=a_{2k-1,2j-1}\bar\alpha_k-a_{2k,2j-1}\beta_k, \quad f_{kj}=a_{2k-1,2j}\bar\alpha_k-a_{2k,2j}\beta_k, \\
&g_{kj}=a_{2k-1,2j-1}\bar \beta_k+a_{2k,2j-1}\alpha_k, \quad h_{kj}=a_{2k-1,2j}\bar\beta_k+a_{2k,2j}\alpha_k.
\end{align*}
Then, we have
\begin{small}
\begin{align}\label{B.26p5}
&\E\(\sum_{j=1}^{n}\E_{j-1}\left|\sum_{k=1}^{j-1}\tr\tilde{a}_{kj}x_jx_k^*\right|^2\)^{p/2}=\E\(\sum_{j=1}^{n}\E_{j-1}\left|\sum_{k=1}^{j-1}\tr x_jx_k^*\tilde{a}_{kj}\right|^2\)^{p/2}\notag\\
=&\E\(\sum_{j=1}^{n}\E_{j-1}\left|\sum_{k=1}^{j-1}e_{kj}\alpha_j-f_{kj}\bar\beta_{j}+g_{jk}\beta_{j}+h_{jk}\bar\alpha_{j}\right|^2\)^{p/2}\notag\\
\leq&C\E\(\sum_{j=1}^{n}\(\left|\sum_{k=1}^{j-1}e_{kj}\right|^2+\left|\sum_{k=1}^{j-1}f_{kj}\right|^2+\left|\sum_{k=1}^{j-1}g_{kj}\right|^2+\left|\sum_{k=1}^{j-1}e_{kj}\right|^2\)\)^{p/2}\notag\\
\leq&C\E\Bigg(\sum_{j=1}^{n}\bigg(\left|\E_{j-1}\sum_{k=1}^{n}e_{kj}\right|^2+\left|\E_{j-1}\sum_{k=1}^{n}f_{kj}\right|^2 \notag\\
+&\left|\E_{j-1}\sum_{k=1}^{n}g_{kj}\right|^2+\left|\E_{j-1}\sum_{k=1}^{n}e_{kj}\right|^2\bigg)\Bigg)^{p/2}\notag\\
\leq&C\E\(\sum_{j=1}^{n}\E_{j-1}\bigg(\left|\sum_{k=1}^{n}e_{kj}\right|^2+\left|\sum_{k=1}^{n}f_{kj}\right|^2+\left|\sum_{k=1}^{n}g_{kj}\right|^2+\left|\sum_{k=1}^{n}e_{kj}\right|^2\bigg)\)^{p/2}.
\end{align}
\end{small}
Using Lemma \ref{B.27} with $q=1$ and the induction hypothesis with ${\mathbf A}$ replaced by ${\mathbf A}{\mathbf A}^*$, together with the fact that $\tr\({\mathbf A}^*{\mathbf A}\)^2\leq \(\tr{\mathbf A}^*{\mathbf A}\)^2,$ we shall get
\begin{align}\label{B.26p6}
&\E\(\sum_{j=1}^{n}\E_{j-1}\bigg(\left|\sum_{k=1}^{n}e_{kj}\right|^2+\left|\sum_{k=1}^{n}f_{kj}\right|^2+\left|\sum_{k=1}^{n}g_{kj}\right|^2+\left|\sum_{k=1}^{n}e_{kj}\right|^2\bigg)\)^{p/2}\notag\\
\leq&C\E\(\sum_{j=1}^{n}\bigg(\left|\sum_{k=1}^{n}e_{kj}\right|^2+\left|\sum_{k=1}^{n}f_{kj}\right|^2+\left|\sum_{k=1}^{n}g_{kj}\right|^2+\left|\sum_{k=1}^{n}e_{kj}\right|^2\bigg)\)^{p/2}\notag\\
=&C\E\(\tr{\mathbf X}^*{\mathbf A}{\mathbf A}^*{\mathbf X}\)^{p/2}\notag\\
\leq&C\(\(\tr{\mathbf A}{\mathbf A}^*\)^{p/2}+\E\left|\tr{\mathbf X}^*{\mathbf A}{\mathbf A}^*{\mathbf X}-\tr{\mathbf A}{\mathbf A}^*\right|^{p/2}\)\notag\\
\leq&C\(\(\tr{\mathbf A}{\mathbf A}^*\)^{p/2}+\(\phi_4\tr\({\mathbf A}{\mathbf A}^*\)^2\)^{p/4}+\phi_p\tr\({\mathbf A}{\mathbf A}^*\)^{p/2}\)\notag\\
\leq&C\(\(\phi_4\tr\({\mathbf A}{\mathbf A}^*\)\)^{p/2}+\phi_{2p}\tr\({\mathbf A}{\mathbf A}^*\)^{p/2}\).
\end{align}
Moreover, using Lemma \ref{2.13} and Lemma \ref{A.13}, we have
\begin{align}\label{B.26p7}
&\sum_{j=1}^n\E\left|\sum_{k=1}^{j-1}\tr\tilde{a}_{kj}x_jx_k^*\right|^p=\sum_{j=1}^n\E\left|\sum_{k=1}^{j-1}\tr x_jx_k^*\tilde{a}_{kj}\right|^p\notag\\
\leq&C\sum_{j=1}^{n}\phi_p\E\(\left|\sum_{k=1}^{j-1}e_{kj}\right|^p+\left|\sum_{k=1}^{j-1}f_{kj}\right|^p+\left|\sum_{k=1}^{j-1}g_{kj}\right|^p+\left|\sum_{k=1}^{j-1}e_{kj}\right|^p\)\notag\\
\leq&C\sum_{j=1}^{n}\phi_p\Bigg(\E\bigg(\Big(\sum_{k=1}^{j-1}\E_{k-1}\left|e_{kj}\right|^2\Big)^{p/2}+\Big(\sum_{k=1}^{j-1}\E_{k-1}\left|f_{kj}\right|^2\Big)^{p/2}\notag\\
&+\Big(\sum_{k=1}^{j-1}\E_{k-1}\left|g_{kj}\right|^2\Big)^{p/2}+\Big(\sum_{k=1}^{j-1}\E_{k-1}\left|h_{kj}\right|^2\Big)^{p/2}\bigg)\notag\\
&+\E\sum_{k=1}^{j-1}\left|e_{kj}\right|^p+\E\sum_{k=1}^{j-1}\left|f_{kj}\right|^p+\E\sum_{k=1}^{j-1}\left|g_{kj}\right|^p+\E\sum_{k=1}^{j-1}\left|h_{kj}\right|^p\Bigg)\notag\\
\leq&C\sum_{j=1}^{n}\phi_p\Bigg(\bigg(\sum_{k=1}^n\lj a_{2k-1,2j-1}\rj^2+\lj a_{2k-1,2j}\rj^2+\lj a_{2k,2j-1}\rj^2+\lj a_{2k,2j}\rj^2\bigg)^{p/2}\notag\\
&+\phi_p\bigg(\sum_{k=1}^n\lj a_{2k-1,2j-1}\rj^p+\lj a_{2k-1,2j}\rj^p+\lj a_{2k,2j-1}\rj^p+\lj a_{2k,2j}\rj^p\bigg)\Bigg)\notag\\
\leq&C\(\phi_p\sum_{j=1}^n\Big(({\mathbf A}^*{\mathbf A})_{2j-1,2j-1}+({\mathbf A}^*{\mathbf A})_{2j,2j}\Big)^{p/2}+\phi_{2p}\tr\({\mathbf A}^*{\mathbf A}\)^{p/2}\)\notag\\
\leq&C\phi_{2p}\tr\({\mathbf A}^*{\mathbf A}\)^{p/2}.
\end{align}
Then combining (\ref{B.26p3}), (\ref{B.26p4}), (\ref{B.26p5}), (\ref{B.26p6}), and (\ref{B.26p7}), the proof of this lemma is complete.
\end{proof}

\begin{lemma}\label{8.7}
  If the conditions of Theorem \ref{mth2} hold, for all $l=1,2 \cdots$, by assuming $$v>C_J={\rm max}\{\sqrt{2C_{J_1}}n^{-1/2}, \cdots, \sqrt[2l]{2C_{J_l}}n^{-1/2}\},\lj t_n(z)\rj\leq 1, $$ where $C_{J_1}, \cdots, C_{J_l}$ are constants that will appear in our proof, we have
  $$\E\lj s_n(z)-\E s_n(z)\rj^{2l} \leq 2C_Jn^{-2l}v^{-4l}\(v+\Delta\)^l.$$
\end{lemma}
\begin{proof}
  Denote by $\E_k$ the conditional expectation given $\{x_{pq}, p,q > k\}$ and let
  \begin{align*}
        \theta_k=\tr{\mb D}-\tr{\mb D}_k,
  \end{align*}
    \begin{align*}
        \gamma_k=\E_{k-1}n^{-1}\tr{\mb D}-\E_k n^{-1}\tr{\mb D}=n^{-1}\E_{k-1}\theta_k-n^{-1}\E_k\theta_k,
  \end{align*}
  then we have
  \begin{align*}
        s_n(z)-\E s_n(z)=\frac{1}{2}\sum_{k=1}^n \gamma_k.
  \end{align*}
  By Lemma \ref{A.1.12}, we have
  \begin{align}\label{theta}
  \lj\theta_k\rj\leq 2v^{-1}, \quad \lj\gamma_k\rj\leq 4n^{-1}v^{-1}.
  \end{align}
  Furthermore, using Lemma \ref{inv}, we obtain
  \begin{align}\label{theta}
  \theta_k=-\tr\(\(\mathbf I_2+n^{-1}{\mb Q}_k^*{\mb P}_k{\mb Q}_k\)\cdot\xi_k(z)\).
  \end{align}
  Note that $\xi_k(z)=t_n(z)\cdot \mathbf I_2+t_n(z)\xi_k(z)\epsilon_k$, where $\epsilon_k$ is a scalar matrix by Corollary \ref{form}. Combining with the fact that
    \begin{align*}
      &\E_{k-1}\tr{\mb Q}_k^*{\mb P}_k{\mb Q}_k-\E_k\tr{\mb Q}_k^*{\mb P}_k{\mb Q}_k \\
      =&\E_{k-1}\tr{\mb Q}_k^*{\mb P}_k{\mb Q}_k-\E_{k-1}\E_k\tr{\mb Q}_k^*{\mb P}_k{\mb Q}_k\\
      =&\E_{k-1}\(\tr{\mb Q}_k^*{\mb P}_k{\mb Q}_k-\tr{\mb P}_k\),
  \end{align*}
we may rewrite
\begin{equation*}
\gamma_k=-n^{-2}\E_{k-1} t_n(z)\(\tr{\mb Q}_k^*{\mb P}_k{\mb Q}_k-\tr{\mb P}_k\)+\(2n\)^{-1}\(\E_{k-1}-\E_{k}\)t_n(z)\theta_k\tr\epsilon_k.
\end{equation*}
Since $\lj t_n(z)\rj\leq 1$,
\begin{align*}
    s_n(z)=\int\frac{1}{\lambda-z}dF_n^w(\lambda)=(2n)^{-1}{\rm tr}(\mb W_n-z\mb I_{2n})^{-1}=(2n)^{-1}\tr \mb D,
\end{align*}
we obtain
\begin{align}
\label{gammak}
\E_k\lj\gamma_k\rj^2
\leq & C\(n^{-4}\E_k\lj\tr{\mb Q}_k^*{\mb P}_k{\mb Q}_k-\tr{\mb P}_k\rj^2+n^{-2}v^{-2}\E_k\lj\tr\epsilon_k\rj^2\)\notag\\
\leq &C\bigg(n^{-4}\E_k\lj\tr{\mb Q}_k^*{\mb P}_k{\mb Q}_k-\tr{\mb P}_k\rj^2+n^{-2}v^{-2}\Big(n^{-1}\E\|x_{kk}\|^2\notag\\
     +&\E_k\lj n^{-1}\tr{\mathbf D}-n^{-1}\tr{\mathbf D}_k\rj^2+\E_k\lj n^{-1}\tr{\mathbf D}_k-n^{-1}\tr{\mb Q}_k^*{\mb D}_k{\mb Q}_k\rj^2 \notag\\
     +&\E_k\lj s_n(z)-\E s_n(z)\rj^2 \Big)\bigg).
\end{align}
For any complex matrix ${\mathbf A}$, we use the notation $\lj\mathbf A\rj^2:={\mathbf A}{\mathbf A}^*$. Write
\begin{align*}
\Psi_k=&n^{-4}\E_k\lj\tr{\mb Q}_k^*{\mb P}_k{\mb Q}_k-\tr{\mb P}_k\rj^2+n^{-2}v^{-2}\Big(n^{-1}\E\|x_{kk}\|^2\\
     +&\E_k\lj n^{-1}\tr{\mathbf D}-n^{-1}\tr{\mathbf D}_k\rj^2+\E_k\lj n^{-1}\tr{\mathbf D}_k-n^{-1}\tr{\mb Q}_k^*{\mb D}_k{\mb Q}_k\rj^2\Big).
\end{align*}
Then, by Lemma \ref{B.22} and Lemma \ref{sz}, we have
\begin{align}\label{moment1}
\E\tr\lj{\mathbf D}_k\rj^2=& \E\sum_{j=1}^{2n-2}\frac{1}{|\lambda_j^{(k)}-z|^2}=\E v^{-1}\Im\{\tr{\mathbf D}_k\}\notag\\
=&\E v^{-1}\Im\{\tr{\mathbf D}_k-\tr{\mathbf D}\}+2n\E v^{-1}\Im\{\(2n\)^{-1}\tr{\mathbf D} \}\notag\\
\leq & 2v^{-2}+2nv^{-1}\E\Im s_n(z)\notag\\
\leq & 2v^{-2}+2 nv^{-1}\(\lj\E s_n(z)-s(z)\rj+\lj s(z)\rj\)\notag\\
\leq & 2v^{-2}+2\pi nv^{-2}\(v+\Delta\),
\end{align}
\begin{align}\label{moment2}
\E\tr\lj{\mathbf D}_k\rj^4
\leq v^{-2}\E\tr\lj{\mathbf D}_k\rj^2\leq & 2v^{-4}+2\pi nv^{-4}\(v+\Delta\),
\end{align}
\begin{align}\label{moment3}
\Psi_k\leq C\(n^{-3}v^{-3}\Im s_n(z)+n^{-3}v^{-2}+n^{-4}v^{-4}\).
\end{align}
By Lemma \ref{2.12} and Lemma \ref{B.26}, together with  inequalities (\ref{gammak}), (\ref{moment1}), (\ref{moment2}) and (\ref{moment3}), for large $n$ we have
\begin{align}\label{m1}
&\E\lj s_n(z)-\E s_n(z)\rj^2
\leq C\sum_{k=1}^n \E\(\E_k\lj\gamma_k\rj^2\)\notag\\
\leq & C_{J_1}\(n^{-2}v^{-4}\(v+\Delta\)+n^{-1}v^{-2}\E\lj s_n(z)-\E s_n(z)\rj^2\).
\end{align}
Furthermore, assuming $v>\sqrt{2C_{J_1}}n^{-1/2}$, by inequality (\ref{m1}), we have
$$\E\lj s_n(z)-\E s_n(z)\rj^2\leq 2C_{J1}n^{-2}v^{-4}\(v+\Delta\).$$
The lemma then follows if $l=1$.

Then, by Lemma \ref{2.13}, Lemma \ref{B.26}, together with inequalities (\ref{gammak}), (\ref{moment1}) , (\ref{moment2}) and (\ref{moment3}), for large $n$ and every $l>1$ we have
\begin{small}
\begin{align}\label{m2}
&\E\lj s_n(z)-\E s_n(z)\rj^{2l} \notag\\
\leq & C\(\E\(\sum_{k=1}^n \(\E_k\lj\gamma_k\rj^2\)\)^l+\sum_{k=1}^n\E\lj\gamma_k\rj^{2l}\)\notag\\
\leq & C\Bigg(\E\(\sum_{k=1}^n\Psi_k\)^l+n^{-2l}v^{-2l}\E\(\sum_{k=1}^n\E_k\lj s_n(z)-\E s_n(z)\rj^2\)^l+ n^{-2l+1}v^{-2l} \Bigg)\notag\\
\leq & C\(\E\(\sum_{k=1}^n\Psi_k\)^l+n^{-l}v^{-2l}\E\lj s_n(z)-\E s_n(z)\rj^{2l}+n^{-2l+1}v^{-2l}\)\notag\\
\leq & C\(n^{-2l}v^{-3l}\E\(\Im s_n(z)\)^l+n^{-l}v^{-2l}\E\lj s_n(z)-\E s_n(z)\rj^{2l}\)\notag\\
\leq & C\Bigg(n^{-2l}v^{-3l}\(\E\lj s_n(z)-\E s_n(z)\rj^{l}+\(\E\Im s_n(z)\)^l\)\notag\\
     &+n^{-l}v^{-2l}\E\lj s_n(z)-\E s_n(z)\rj^{2l}\Bigg)\notag\\
\leq & C_{J_l}\Bigg(n^{-2l}v^{-3l}\E\lj s_n(z)-\E s_n(z)\rj^{l}+n^{-2l}v^{-4l}\(v+\Delta\)^l\notag\\
     &+n^{-l}v^{-2l}\E\lj s_n(z)-\E s_n(z)\rj^{2l}\Bigg).
\end{align}
\end{small}
Assuming $v>\sqrt[l]{2C_{J_l}}n^{-1/2}$, by induction and inequality (\ref{m2}), we have
$$\E\lj s_n(z)-\E s_n(z)\rj^{2l}\leq 2C_{Jl}n^{-2l}v^{-4l}\(v+\Delta\)^l.$$
Now, the proof of this lemma is complete.
\end{proof}

\begin{lemma}\label{8.6}
 If the conditions of Theorem \ref{mth2} and Lemma \ref{8.7} hold and
 $$C_\epsilon={\max}\{2,C_{\epsilon_2},C_{\epsilon_3},C_{\epsilon_4}\}, $$
 we have the following estimations:
 \begin{enumerate}
   \item $\left|\E\tr\epsilon_k\right|\leq C_\epsilon n^{-1}v^{-1},$
   \item $\E\left|\tr\(\epsilon_k^2\)\right| \leq C_\epsilon n^{-1}v^{-2}\(v+\Delta\),$
   \item $\E\left|\tr\(\epsilon_k^3\)\right| \leq C_\epsilon n^{-1}v^{-2}\(\(v+\Delta\)+\(v+\Delta\)^2\),$
   \item $\E\left|\tr\(\epsilon_k^4\)\right| \leq C_\epsilon n^{-1}v^{-2}\(n^{-1/2}\(v+\Delta\)+\(v+\Delta\)^2\).$
 \end{enumerate}
 \end{lemma}
 \begin{proof}
Firstly, by Lemma \ref{A.1.12}, we have
\begin{equation*}
  \lj\E\tr\epsilon_k\rj=n^{-1}\lj\E\(\tr {\mb D}-\tr {\mb D}_k\)\rj\leq 2n^{-1}v^{-1}.
\end{equation*}
This completes the proof of (1).

Then, by Lemma \ref{B.26} and Lemma \ref{8.7}, together with inequality (\ref{moment1}) we have
\begin{align*}
  \E\lj\tr\epsilon_k^2\rj=&2^{-1}\E\lj\tr\epsilon_k\rj^2\\
  \leq &C\Bigg(n^{-1}\E\| x_{kk}\|^2+n^{-2}\E\lj\tr{\mb Q}_k^*{\mb D}_k{\mb Q}_k-\tr{\mb D}_k\rj^2\\
  +&\E\lj s_n(z)-\E s_n(z)\rj^2+n^{-2}v^{-2}\Bigg)\\
  \leq &C_{\epsilon_2}n^{-1}v^{-2}\(v+\Delta\).
\end{align*}
This completes the proof of (2).

Furthermore, since $\|x_{kk}\|$ are truncated at $n^{1/4}$, by Lemma \ref{B.26} and Lemma \ref{8.7}, together with inequalities (\ref{moment1}) and (\ref{moment2}), we shall have the following estimates:
\begin{align}\label{tr1}
n^{-2}\E \|x_{kk}\|^4 \leq n^{-2}n^{1/2}\sigma^2=n^{-3/2}\sigma^2,
\end{align}
\begin{align}\label{tr2}
     &n^{-4}\E \lj\tr{\mb Q}_k^*{\mb D}_k{\mb Q}_k-\tr{\mb D}_k\rj^4\notag\\
\leq &Cn^{-4}\E\(\phi_8\tr\lj{\mathbf D}_k\rj^4+\phi_4^2\(\tr\lj{\mathbf D}_k\rj^2\)^2\)\notag\\
\leq &Cn^{-4}\(\phi_6n^{1/2}\E\tr\lj{\mathbf D}_k\rj^4+\phi_4^2\E\(v^{-4}+n^2v^{-2}\(\Im s_n(z)\)^2\)\) \notag\\
\leq &Cn^{-4}\Bigg(n^{1/2}\(v^{-4}+nv^{-4}\(v+\Delta\)\)+v^{-4}\notag\\
+    &n^2v^{-2}\Big(\E\lj s_n(z)-\E s_n(z)\rj^2+\big|\lj\E s_n(z)-s(z)\rj+\lj s(z)\rj\big|^2\Big)\Bigg)\notag\\
\leq &Cn^{-1}v^{-2}\(n^{-1/2}\(v+\Delta\)+\(v+\Delta\)^2\).
\end{align}
Then combining Lemma \ref{8.7}, inequalities (\ref{tr1}) and (\ref{tr2}), we have
\begin{align*}
  \E\lj\tr\epsilon_k^4\rj=&8^{-1}\E\lj\tr\epsilon_k\rj^4\\
  \leq &C\Bigg(n^{-2}\E\| x_{kk}\|^4+n^{-4}\E\lj\tr{\mb Q}_k^*{\mb D}_k{\mb Q}_k-\tr{\mb D}_k\rj^4\\
  +&\E\lj s_n(z)-\E s_n(z)\rj^4+n^{-4}v^{-4}\Bigg)\\
  \leq &C_{\epsilon_4}n^{-1}v^{-2}\(n^{-1/2}\(v+\Delta\)+\(v+\Delta\)^2\),
\end{align*}
which completes the proof of (4).

At last, by the elementary inequality, we obtain
\begin{align*}
  \E\lj\tr\epsilon_k^3\rj=&4^{-1}\E\lj\tr\epsilon_k\rj^3  \leq 8^{-1}\(\E\lj\tr\epsilon_k\rj^2+\E\lj\tr\epsilon_k\rj^4\)\\
  \leq &C_{\epsilon_3}n^{-1}v^{-2}\(\(v+\Delta\)+\(v+\Delta\)^2\).
\end{align*}
 \end{proof}

\begin{lemma}[Lemma 8.8 in \cite{bai2010spectral}]\label{delta}
 If $\lj\delta_n(z)\rj\leq v<3^{-1}$ for all $z=u+iv$, then $\Delta \leq C_\delta v$.
\end{lemma}
\begin{remark}
Lemma \ref{delta} was proved in order to get a convergence rate of Wigner matrix. According to the result of \cite{yinbai2013e}, combining with Lemma \ref{bai93in}, we can easily obtain that this lemma is also right in the case of quaternion self-dual Hermitian matrices.
\end{remark}

\section{Proof of Theorem \ref{mth1}}

Now, we are in position to present the proof of our main theorems.
\begin{proof}[Proof of Theorem \ref{mth2}]
Yin, Bai and Hu in \cite{yinbai2013} derived
\begin{align}\label{esn}
  \E s_n(z)=-\frac{1}{2}\(z-\delta_n(z)-\sqrt{\(z+\delta_n(z)\)^2-4}\),
\end{align}
where
\begin{equation}\label{deltap}
  \delta_n(z)=\(2n\)^{-1}\sum_{k=1}^n\E\tr\Bigg(\bigg(z+\E s_n(z)\bigg)^{-1}\cdot\epsilon_k\cdot\bigg(\Big(z+\E s_n(z)\Big)\cdot{\mathbf I}_2-\epsilon_k\bigg)^{-1}\Bigg).
\end{equation}
By equation (\ref{deltap}) and equation $\xi_k(z)=t_n(z)\cdot\mb I_2+t_n(z)\xi_k\cdot\epsilon_k$, we have
\begin{align}\label{deltaab}
\lj\delta_n(z)\rj=&\lj (2n)^{-1}\sum_{k=1}^n t_n(z) \E\tr\(\epsilon_k\xi_k\)\rj\notag\\
           \leq &(2n)^{-1}\sum_{k=1}^n\Big(\lj t_n(z)\rj^2\lj\E\tr\epsilon_k\rj+\lj t_n(z)\rj^3\E\lj\tr\(\epsilon_k\)^2\rj\notag\\
           +    &\lj t_n(z)\rj^4\E\lj\tr\(\epsilon_k\)^3\rj+v^{-1}\lj t_n(z)\rj^4\E\lj\tr\(\epsilon_k\)^4\rj\Big).
\end{align}
Here we use the fact that
 \begin{align*}
\Im \{-z-\frac{1}{n}\varpi_k^*\mb D_k{\varpi _k}\}=  - v\({1 + \frac{1}{n}\varpi _k^*{\mb D_k\mb D_k^*}{\varpi_k}}\)
 <  - v.
\end{align*}

Due to the semicircular law, $\lj\delta_n(z)\rj<v$ when $v=3^{-1}$ and $n$ large enough.
We claim that for all large $n$ and any $C_v n^{-1/2}\leq v\leq 3^{-1}$, where $C_v=\max\{C_J,$ $ \sqrt{\(C_\epsilon+1\)\(2C_\delta^2+7C_\delta+6\)}\}$, $\lj\delta_n(z)\rj<v$.

If that is not the case, since the function $\lj\delta_n(z)\rj$ is continuous in the region $v\geq n^{-1/2}$, there must exist a $z_0=u_0+v_0i$ satisfying $\lj\delta_n(z_0)\rj=v_0$. Thus we have $\Im\{z_0+\delta_n(z_0)\}\geq 0$. It is well known, if a complex number $Z$ satisfies $\Im Z \neq0$ and $\Im Z=-\Im\{\frac{1}{Z}\}$, then $\lj Z\rj=1$. Since $t_n(z)=-\frac{1}{t_n(z)}+z+\delta_n(z), \Im \{z+\E s_n(z)\}\geq \Im z>0$, this implies that $\lj t_n(z_0)\rj=1$ if $\Im\{z_0+\delta_n(z_0)\}= 0$. Otherwise, if $\Im\{z_0+\delta_n(z_0)\}>0$, by Lemma \ref{semst} and equation (\ref{esn}) we have $t_n(z_0)=-\E s_n(z_0)+\delta_n(z_0)=-s(z_0+\delta_n(z_0))$, which implies that $\lj t_n(z_0)\rj<1$ due to Lemma \ref{sz}. To sum up the above arguments, we will have $\lj t_n(z_0)\rj\leq 1$ if $\lj\delta_n(z_0)\rj=v_0$.

Then combining inequality (\ref{deltaab}), Lemma \ref{delta} and Lemma \ref{8.6}, we have
 \begin{align*}
\lj\delta_n(z_0)\rj< \(C_\epsilon+1\)\(2C_\delta^2+7C_\delta+6\)<v_0,
\end{align*}
which is a contradiction.

Applying Lemma \ref{delta}, we obtain that
$\|\E F_n-F\|_d=O(n^{-1/2})$.

Finally, using Lemma \ref{8.7} we have
\begin{enumerate}
         \item for $v=n^{-2/5}$,
         \begin{align}\label{p}
           \int_{-16}^{16}\E\lj s_n(z)-\E s_n(z)\rj du &\leq \int_{-16}^{16}\(\E\lj s_n(z)-\E s_n(z)\rj^2\)^{1/2} du  \notag \\
                                                       &\leq Cn^{-1}v^{-3/2}\leq v.
         \end{align}
         \item for $v=n^{-2/5+\eta}$, choose $l>(5\eta)^{-1}$. Then for any $\epsilon>0$
         \begin{align}\label{as}
           &{\rm P}\(\int_{-16}^{16}\lj s_n(z)-\E s_n(z)\rj du \geq \epsilon v \) \notag\\
           &\leq (32/\epsilon)^{2l}\E\lj s_n(z)-\E s_n(z)\rj^{2l}=Cn^{-5l\eta},
         \end{align}
         which is summable.
\end{enumerate}
Combining (\ref{p}) and (\ref{as}), we complete the proof of this theorem.
\end{proof}

\begin{proof}[Proof of Theorem \ref{mth1}]
Let $\widetilde{\bf S}_n=(\frac{1}{\sqrt n}\widetilde y_{jk})_{n\times n}$, $\widehat{\bf S}_n=(\frac{1}{\sqrt n}\widehat y_{jk})_{n\times n}$, where
$$\widetilde y_{jk}=y_{jk}I(\|y_{jk}\|\leq n^{1/4}), \quad \widehat y_{jk}=y_{jk}I(\|y_{jk}\|\leq n^{1/4})-\E y_{jk}I(\|y_{jk}\|\leq n^{1/4}).$$
Let $F_n^s$, $\widetilde F_n^s$, $\widehat F_n^s$ and $F_n^w$ denote the ESD of $\mb S_n$, $\widetilde {\mb S}_n$, $\widehat {\mb S}_n$ and $\mb W_n$, respectively.
Firstly, by Lemma \ref{A.43}, we have
\begin{align*}
\|F_n^s-\widetilde F_n^s\|_d &\leq (2n)^{-1}{\rm rank}\(\(\mb S_n-\widetilde {\mb S}_n\)\) \notag\\
&\leq n^{-1}\sum_{j\neq k}{\rm I}\(\|y_{jk}\|> n^{1/4}\)+n^{-1}\sum_{j=1}^n{\rm I}\(\|y_{jj}\|> n^{1/4}\)
\end{align*}
Since
 \begin{align*}
&\E \(n^{-1}\sum_{j\neq k}{\rm I}\(\|y_{jk}\|> n^{1/4}\)+n^{-1}\sum_{j=1}^n{\rm I}\(\|y_{jj}\|> n^{1/4}\)\) \notag\\
                               &\leq n^{-1}\sum_{j\neq k}{\rm P}\(\|y_{jk}\|> n^{1/4}\)+n^{-1}\sum_{j=1}^n{\rm P}\(\|y_{jj}\|> n^{1/4}\) \notag\\
                               &\leq n^{-5/2}\sum_{j\neq k}\E\|y_{jk}\|^6+n^{-3/2}\sum_{j=1}^n\E\|y_{jj}\|^2\leq 2Mn^{-1/2},
\end{align*}
and
 \begin{align*}
&{\rm Var}\(n^{-1}\sum_{j\neq k}{\rm I}\(\|y_{jk}\|> n^{1/4}\)+n^{-1}\sum_{j=1}^n{\rm I}\(\|y_{jj}\|> n^{1/4}\)\) \notag\\ &\leq n^{-2}\sum_{j\neq k}{\rm P}\(\|y_{jk}\|> n^{1/4}\)+n^{-2}\sum_{j=1}^n{\rm P}\(\|y_{jj}\|> n^{1/4}\) \notag\\
                               &\leq 2Mn^{-3/2}.
\end{align*}
By Bernstein's inequality, for any $\zeta>0$ we obtain
 \begin{align*}
{\rm P}\(n^{1/2-\zeta}\|F_n^s-\widetilde F_n^s\|_d\geq \varepsilon\)\leq 2e^{-4^{-1}\varepsilon n^{1/2+\zeta}},
\end{align*}
which immediately implies that
 \begin{align}\label{11}
n^{1/2-\zeta}\|F_n^s-\widetilde F_n^s\|_d \to 0, \quad a.s..
\end{align}

Then, by Lemma \ref{A.41}, we have
 \begin{align}\label{12}
L^3\(\widetilde F_n^s,\widehat F_n^s\) &\leq (2n)^{-1}\tr\(\widetilde {\mb S}_n-\widehat {\mb S}_n\)\(\widetilde {\mb S}_n-\widehat {\mb S}_n\)^* \notag\\
                               &=n^{-2}\sum_{j\neq k}\|\E y_{jk}I(\|y_{jk}\|\leq n^{1/4})\|^2 \notag\\
                               &+n^{-2}\sum_{j=1}^n\|\E y_{jj}I(\|y_{jj}\|\leq n^{1/4})\|^2 \notag\\
                               &\leq n^{-9/2}\sum_{j\neq k}\E^2\|y_{jk}\|^6+n^{-5/2}\sum_{j=1}^n\E^2\|y_{jj}\|^2\leq 2M^2n^{-3/2}.
\end{align}
Furthermore, by Lemma \ref{A.41} we have
 \begin{align*}
&L^3\(\widehat F_n^s, F_n^w\)\leq (2n)^{-1}\tr\(\widehat {\mb S}_n-{\mb W}_n\)\(\widehat {\mb S}_n-{\mb W}_n\)^* \notag\\
                               &=n^{-2}\sum_{j\neq k}\(1-\sigma_{jk}^{-1}\)^2\lj y_{jk}I(\|y_{jk}\|\leq n^{1/4})-\E y_{jk}I(\|y_{jk}\|\leq n^{1/4})\rj^2 \notag\\
                               &+n^{-2}\sum_{j=1}^n\(1-\sigma\sigma_{jj}^{-1}\)^2\lj y_{jj}I(\|y_{jj}\|\leq n^{1/4})-\E y_{jj}I(\|y_{jj}\|\leq n^{1/4})\rj^2 .
\end{align*}
Also, we have
 \begin{align*}
\(1-\sigma_{jk}^2\)^2 &\leq\(\E \|y_{jk}\|^2I(\|y_{jk}\|> n^{1/4})+\E^2 \|y_{jk}\|I(\|y_{jk}\|> n^{1/4})\)^2\\
                      &\leq 4\E^2 \|y_{jk}\|^2I(\|y_{jk}\|> n^{1/4})\leq 4M^2n^{-2},
\end{align*}
and
 \begin{align*}
\(\sigma^2-\sigma_{jj}^2\)^2 &\leq\(\E \|y_{jj}\|^2I(\|y_{jj}\|> n^{1/4})+\E^2 \|y_{jj}\|I(\|y_{jj}\|> n^{1/4})\)^2\\
                      &\leq 4\E^2 \|y_{jj}\|^2I(\|y_{jk}\|> n^{1/4})\leq 4M^2n^{-1/2},
\end{align*}
which immediately imply,
 \begin{align*}
&\E \Bigg(n^{-2}\sum_{j\neq k}\(1-\sigma_{jk}^{-1}\)^2\lj y_{jk}I(\|y_{jk}\|\leq n^{1/4})-\E y_{jk}I(\|y_{jk}\|\leq n^{1/4})\rj^2 \notag\\
                               &+n^{-2}\sum_{j=1}^n\(1-\sigma\sigma_{jj}^{-1}\)^2\lj y_{jj}I(\|y_{jj}\|\leq n^{1/4})-\E y_{jj}I(\|y_{jj}\|\leq n^{1/4})\rj^2 \Bigg) \\
                               &\leq n^{-2}\sum_{j\neq k}\(1-\sigma_{jk}^2\)^2+\sigma^{-2}n^{-2}\sum_{j=1}^n\(\sigma^2-\sigma_{jj}^2\)^2 \\
                               &\leq 4\(M^2+\(\frac{M}{\sigma}\)^2\)n^{-3/2},
\end{align*}
and for any $\zeta>0$
 \begin{align*}
&{\rm Var} \Bigg(n^{-1/2-\zeta}\sum_{j\neq k}\(1-\sigma_{jk}^{-1}\)^2\lj y_{jk}I(\|y_{jk}\|\leq n^{1/4})-\E y_{jk}I(\|y_{jk}\|\leq n^{1/4})\rj^2 \notag\\
                               &+n^{-1/2-\zeta}\sum_{j=1}^n\(1-\sigma\sigma_{jj}^{-1}\)^2\lj y_{jj}I(\|y_{jj}\|\leq n^{1/4})-\E y_{jj}I(\|y_{jj}\|\leq n^{1/4})\rj^2 \Bigg) \\
                               &\leq 16M^2n^{-1-2\zeta}\(\sum_{j\neq k}\(1-\sigma_{jk}^2\)^4+\sigma^{-4}\sum_{j=1}^n\(\sigma^2-\sigma_{jj}^2\)^4\) \notag\\
                               &\leq 16\(1+16\(\frac{M}{\sigma}\)^4\)n^{-1-2\zeta}.
\end{align*}
Thus we obtain
 \begin{align}\label{13}
n^{1/2-\zeta}L^3\(\widehat F_n^s, F_n^w\)\to 0, \quad a.s..
\end{align}
Finally, by inequalities (\ref{11}), (\ref{12}), (\ref{13}) and Theorem \ref{mth2}, the proof of Theorem \ref{mth1} is complete.
\end{proof}


\end{document}